\documentclass[11pt]{article}

\usepackage{enumerate}
\usepackage{amsmath}
\usepackage{amssymb,latexsym}
\usepackage{amsthm}
\usepackage{color}
\usepackage{graphicx}
\usepackage{mathabx}
\usepackage{hyperref}
\usepackage{amscd}

\title{Scalar $q$-subresultants and Dickson matrices}
\author{Bence Csajb\'ok\thanks{Supported by the \'UNKP-18-4 New National Excellence Program of the Ministry of Human Capacities and by OTKA Grant No. K 124950. }}
\date{}

\newcommand{\F}{{\mathbb F}}

\newcommand{\K}{{\mathbb K}}

\newcommand{\la}{\langle}
\newcommand{\ra}{\rangle}

\newtheorem{theorem}{Theorem}[section]

\newtheorem{lemma}[theorem]{Lemma}
\newtheorem{corollary}[theorem]{Corollary}

\newtheorem{proposition}[theorem]{Proposition}
\newtheorem{result}[theorem]{Result}

\DeclareMathOperator{\PG}{{PG}}

\begin{document}
\maketitle

\begin{abstract}
Following the ideas of Ore and Li we study $q$-analogues of scalar subresultants and show how these results can be applied to determine the rank of an $\F_q$-linear transformation $f$ of $\F_{q^n}$. As an application we show how certain minors of the Dickson matrix $D(f)$, associated with $f$, determine the rank of $D(f)$ and hence the rank of $f$. 
 \end{abstract}



\section{Introduction}

Let $f(x)=\sum_{i=0}^{k} a_i x^i$ and $g(x)=\sum_{i=0}^l b_i x^i$,  with $a_kb_l\neq 0$, be two univariate polynomials with coefficients in the field $\K$ \footnote{Note that in many of the cited literature $a_0$ and $b_0$ are used to denote the leading coefficients of $f$ and $g$.}. In elimination theory, the classical resultant of $f$ and $g$ is 
\[\mbox{Res}(f,g)=(-1)^{kl}b_l^k\prod_{i=1}^lf(\xi_i),\]
where $g(x)=b_l\prod_{i=1}^l(x-\xi_i)$ with $\xi_1,\xi_2,\ldots,\xi_l \in \overline{\K}$ (where $\overline{\K}$ denotes the algebraic closure of $\K$).
For $0\leq m\leq \min\{k,l\}$ consider the following $(k+l-2m)\times (k+l-2m)$ matrix:
\[R_m(f,g):=
\begin{pmatrix}
a_k & a_{k-1}&  a_{k-2} & \ldots & a_{k-l+m+1}&\ldots & a_{2m-l+2} & a_{2m-l+1} \\
0 & a_k & a_{k-1}&\ldots & a_{k-l+m+2} & \ldots &a_{2m-l+3} & a_{2m-l+2} \\
\vdots & \vdots & \vdots & \vdots & \vdots & \vdots & \vdots & \vdots\\
0 & \ldots & 0 & \ldots & a_k & \ldots &a_{m+1}& a_{m} \\
b_l & b_{l-1}&  b_{l-2} & \ldots &\ldots& \ldots &b_{2m-k+2} & b_{2m-k+1} \\
0 & b_l & b_{l-1}&\ldots & \ldots& \ldots &b_{2m-k+3} & b_{2m-k+2} \\
\vdots & \vdots & \vdots & \vdots & \vdots & \vdots & \vdots & \vdots \\
0 & \ldots & 0 & b_l & \ldots &\ldots &b_{m+1} & b_{m} \\
\end{pmatrix}
,\]
where coefficients out of range are considered to be $0$.

The determinant of $R_m(f,g)$ is also called the $m$-th \emph{scalar subresultant} of $f$ and $g$. 
Note that $|R_0(f,g)|=\mbox{Res}(f,g)$ and hence $\gcd(f,g)=1$ if and only if $|R_0(f,g)|\neq 0$. This result has the following well-known generalization in elimination theory. For a proof we cite here the Appendix of \cite{Hetamas} and the references therein, since the proof of Theorem \ref{resoriginal2} was motivated by the arguments found there.

\begin{result}
\label{resoriginal}
The degree of $\gcd(f,g)$ is $t$ if and only if $|R_0(f,g)|=\ldots=|R_{t-1}(f,g)|=0$ and $|R_t(f,g)|\neq 0$.
\end{result}

The strength of the Result \ref{resoriginal} is that it provides a way to study the number of common roots of $f$ and $g$ only by means of their coefficients. 

\medskip

Now let $\K$ be a field of characteristic $p$, and let $q$ be a power of $p$. 
A \emph{$q$-polynomial} over $\K$ with \emph{$q$-degree} $m$ is a polynomial of the form $f(x)=\sum_{i=0}^{m}a_ix^{q^i}$, with $a_m\neq 0$ and $a_0,a_1,\ldots,a_m\in \K$. When $q=p$ prime, $q$-analogue of the classical resultant for $q$-polynomials was already mentioned  in \cite[Chapter 1, Section 7]{Ore}, however, an explicit formula was not given there. An explicit formula can be found for example in \cite[page 59]{fa}. 

The subresultant theory was extended to Ore polynomials (cf. \cite{Ore2}) and hence also to the non-commutative ring of $q$-polynomials by Li in \cite{Li}. Here the non-commutative operation between two $q$-polynomials is composition, while addition is defined as usual. Note that this ring is a right-Euclidean domain with respect to the $q$-degree, cf. \cite{Ore}. When $g=f\circ h$ then we will also say that $h$ is a symbolic right divisor of $g$. 
Note that in the paper of Li the word subresultant is used to what is also known as \emph{polynomial subresultant}. In the classical theory the $m$-th scalar subresultant is  the leading coefficient of the $m$-th polynomial subresultant. See for example \cite[Section 2]{multivar} for a brief summary, where $S_m^{(m)}$ corresponds to what we (and some other authors) call scalar subresultant. For the various notions consult with \cite{sum}. 

Let $\K=\F_{q^n}$ and consider $\K$ as an $n$-dimensional vector space over $\F_q$.  Then there is an isomorphism between the ring of $q$-polynomials 
\[\left\{ \sum_{i=0}^{n-1}a_i x^{q^i} \colon a_0,\ldots,a_{n-1}\in \F_{q^n}\right\} \] considered modulo $(x^{q^n}-x)$ and the ring of $\F_q$-linear transformations of $\F_{q^n}$. The set of roots of a $q$-polynomial form an $\F_q$-subspace and the dimension of this subspace is the dimension of the kernel of the corresponding $\F_q$-linear transformation. Thus $\deg \gcd(f(x),x^{q^n}-x)=q^{n-k}$, where $k$ is the rank of the $\F_q$-linear transformation of $\F_{q^n}$ defined by $f(x)$. When $n$ is clear from the context, then we will say that $k$ is the \emph{rank} of $f$. 

\begin{result}[Ore {\cite[Theorem 2]{Ore}}]
\label{eukl}
The greatest common symbolic right divisor of two $q$-polynomials is the same as their ordinary greatest common divisor.
\end{result}

It follows that the $q$-subresultant theory can be applied to determine $\gcd(f(x),x^{q^n}-x)$ and hence the rank of $f$. 
Our contribution to this theory is a direct proof to a $q$-analogue of Result \ref{resoriginal} providing sufficient and necessary conditions which ensure that $f$ has rank $n-k$ (cf. Theorem \ref{resoriginal2}). 

\medskip

Recall that the Dickson matrix associated with $f(x)=\sum_{i=0}^{n-1}a_i x^{q^i}\in \F_{q^n}[x]$ is
\[D(f):=
\begin{pmatrix}
a_0 & a_1 & \ldots & a_{n-1} \\
a_{n-1}^q & a_0^q & \ldots & a_{n-2}^q \\
\vdots & \vdots & \vdots & \vdots \\
a_1^{q^{n-1}} & a_2^{q^{n-1}} & \ldots & a_0^{q^{n-1}}
\end{pmatrix}
.\]

It is well-known that the rank of $f$ equals the rank of $D(f)$, see for example \cite[Proposition 4.4]{ffa} or \cite[Proposition 5]{GM}.
In some recent constructions of maximum scattered subspaces and MRD-codes it was crucial to the determine the rank of certain Dickson matrices (cf. \cite[Section 7]{withZanella} and \cite[Section 5]{withZullo}). In these papers this was done by considering certain minors of such matrices and excluding the possibility that their determinants vanish at the same time. On the other hand, in \cite[Section 3]{withZhou} Dickson matrices were used to prove non-existence results of certain MRD-codes. This was done by proving that, for a certain choice of the parameters, all $6\times 6$ submatrices of a $9\times 9$ Dickson matrix have zero determinant.
As an application of Theorem \ref{resoriginal2} we show that it is enough to investigate the nullity of the determinant of at most $k+1$ well-defined minors to decide whether $f$ has rank $n-k$. This result can significantly simplify the above mentioned arguments. 

\medskip
To state here the main result of this paper we introduce the notion $D_{m}(f)$ to denote the $(n-m)\times (n-m)$  matrix obtained from $D(f)$ after removing its first $m$ columns and last $m$ rows. Our main result is the following.

\begin{theorem}
	\label{conn}
	$\dim_q (\ker f)=\mu$ if and only if
	\begin{equation}
	|D_0(f)|=|D_1(f)|=\ldots=|D_{\mu-1}(f)|=0
	\end{equation}
	and $|D_{\mu}(f)|\neq 0$.
\end{theorem}

Results in a similar direction have been obtained recently in \cite{teoremone} where for each $q$-polynomial $f$ of $q$-degree $k$, $k$ conditions were given, in terms of the coefficients of $f$, which are satisfied if and only if $f$ has rank $n-k$ (there is a hidden $(k+1)$-th condition here as well, namely the assumption that the coefficient of $x^{q^k}$ in $f$ is non-zero). Independently, in \cite{John} it was proved that the rank of $f$ is $n-m$ if and only if a certain $k\times k$ matrix has rank $k-m$. 
If $m=k$, then this result gives back the main result of \cite{teoremone}. 

\section{\texorpdfstring{Scalar $q$-subresultants}{Scalar q-subresultants}}
\label{qsubres}

Consider $f(x)=\sum_{i=0}^{k}a_i x^{q^i}$ and $g(x)=\sum_{i=0}^l b_i x^{q^i}$, two $q$-polynomials with coefficients in $\widebar{\F_q}$ such that $a_k b_l \neq 0$.
Put
\[q^{\mu}=\deg \gcd(f,g).\]
By Result \ref{eukl}, $\mu$ also equals the $q$-degree of the symbolic greatest common right divisor of $f$ and $g$. 

For $m\leq \min\{k,l\}$ we define the $(k+l-2m)\times (k+l-2m)$ matrix $R_{m,q}(f,g)$ as follows:
\[\begin{pmatrix}
a_k^{q^{l-m-1}} & a_{k-1}^{q^{l-m-1}}&  a_{k-2}^{q^{l-m-1}} & \ldots & a_{k+m-l+1}^{q^{l-m-1}} &\ldots & a_{2m-l+2}^{q^{l-m-1}} & a_{2m-l+1}^{q^{l-m-1}} \\
0 & a_k^{q^{l-m-2}} & a_{k-1}^{q^{l-m-2}}&\ldots& a_{k+m-l+2}^{q^{l-m-2}}&\ldots & a_{2m-l+3}^{q^{l-m-2}} & a_{2m-l+2}^{q^{l-m-2}} \\
\vdots & \vdots & \vdots & \vdots & \vdots & \vdots &\vdots &\vdots\\
0 & \ldots & 0 & \ldots & a_k & \ldots & a_{m+1}& a_{m} \\
b_l^{q^{k-m-1}} & b_{l-1}^{q^{k-m-1}}&  b_{l-2}^{q^{k-m-1}} & \ldots &\ldots &\ldots& b_{2m-k+2}^{q^{k-m-1}} & b_{2m-k+1}^{q^{k-m-1}} \\
0 & b_l^{q^{k-m-2}} & b_{l-1}^{q^{k-m-2}}&\ldots & \ldots & \ldots & b_{2m-k+3}^{q^{k-m-2}} & b_{2m-k+2}^{q^{k-m-2}} \\
\vdots & \vdots & \vdots & \vdots & \vdots &\vdots &\vdots &\vdots\\
0 & \ldots & 0 & b_l & \ldots &\ldots & b_{m+1} &b_{m} \\
\end{pmatrix}
.\]

Note that $R_{m+1,q}(f,g)$ is obtained from $R_{m,q}(f,g)$ by removing its first and last columns, and its first and $(l-m+1)$-th rows.

We state here the $q$-analogue of Result \ref{resoriginal}.

\begin{theorem}
		\label{resoriginal2}
	The $q$-degree of $\gcd(f,g)$ is $\mu$ if and only if $|R_{0,q}(f,g)|=\ldots=|R_{\mu-1,q}(f,g)|=0$ and $|R_{\mu,q}(f,g)|\neq 0$.
\end{theorem}

We prove this result directly by following the proof of the classical Result \ref{resoriginal}. Theorem \ref{resoriginal2} will easily follow from Proposition \ref{prop}.

\begin{proposition}
	Recall $q^{\mu}=\deg \gcd(f,g)$ and let $m\leq \mu$. 
	Let $c(x)=\sum_{i=0}^{k-m} c_i x^{q^i}$ and $d(x)=\sum_{i=0}^{l-m} d_i x^{q^i}$ be $q$-polynomials over $\widebar{\F_q}$ with $c_{k-m}=a_k^{q^{l-m}}$,  $d_{l-m}=b_l^{q^{k-m}}$, and their other coefficients are considered as unknowns. 
	Then the set of solutions for these coefficients such that  
	\begin{equation}
	\label{eq}
	d \circ f - c \circ g=0
	\end{equation}	
form a $(\mu-m)$-dimensional affine $\widebar{\F_q}$-space.
\end{proposition}
\begin{proof}
	First assume that $f$ and $g$ have only simple roots.
	
	Let $r$ be the greatest common monic symbolic right divisor of $f$ and $g$ and suppose that \eqref{eq} holds for some $c$ and $d$. Then 
	$f=f_1 \circ r$ and $g=g_1 \circ r$ and \eqref{eq} yields $d \circ f_1 = c \circ g_1$, thus $d$ is zero on $f_1(\ker g_1)$ (in this proof the kernel is always taken over $\widebar{\F_q}$) and $c$ is zero on $g_1(\ker f_1)$.
	Since the greatest common symbolic right divisor of $f_1$ and $g_1$ is the identity map, it follows that $\gcd(f_1,g_1)=x$ and hence $\ker f_1 \cap \ker g_1=\{0\}$. 
	Thus $\dim_q f_1(\ker g_1)=\dim_q \ker g_1 = l-\mu$ and similarly $\dim_q g_1(\ker f_1)=k-\mu$. It follows that the unique $q$-polynomial $d_1$ of $q$-degree $l-\mu$ and with leading coefficient $b_l^{q^{k-\mu}}$ which vanishes on $f_1(\ker g_1)$ is a divisor of $d$. By  Result \ref{eukl} $\gcd(d,d_1)=d_1$ is also a symbolic right divisor of $d$, i.e. $d=d_2 \circ d_1$, for some monic $d_2$ with $q$-degree $(\mu-m)$. Similarly, the unique $q$-polynomial $c_1$ of $q$-degree $k-\mu$ and with leading coefficient $a_k^{q^{l-\mu}}$ which vanishes on $g_1(\ker f_1)$ is a symbolic right divisor of $c$, i.e. $c= c_2 \circ c_1$, for some monic $c_2$ with $q$-degree $(\mu-m)$.
	
	Note that 
   \[d_1 \circ f_1-c_1 \circ g_1\]
    has $q$-degree $k+l-2\mu-1$ (the coefficient of $x^{q^{k+l-2\mu}}$ vanishes because of the assumptions on the leading coefficients of $c$ and $d$) and it vanishes on $\ker f_1 \oplus \ker g_1$. Thus it is the zero polynomial.
    
    Then 
    \[c_2 \circ c_1 \circ g_1 =c\circ g_1 = d\circ f_1 = d_2 \circ d_1 \circ f_1= d_2 \circ c_1 \circ g_1\]
    and hence $c_2=d_2$. On the other hand, if $c_2=d_2$, then 
	we clearly have a solution since \eqref{eq} becomes $d_2 \circ (d_1 \circ f_1-c_1 \circ g_1) \circ r$ with the zero polynomial in the middle.
	 
	Since we can choose the first $(\mu-m)$ coefficients of $d_2(x)=\sum_{i=0}^{\mu-m}\hat{d}_ix^{q^i}$ arbitrarily, the assertion follows. More precisely, if
	$d_1(x)=\sum_{j=0}^{\mu-m}\bar{d_j}x^{q^j}$ with $\bar{d}_{l-\mu}=b_l^{q^{k-\mu}}$ and with coefficients out of range defined as $0$, then $d(x)$ is of the form
	\[\sum_{i=0}^{k-m}\sum_{j=0}^i \hat{d}_{i-j}\bar{d_j}^{q^{i-j}}x^{q^i},\]
	with $\hat{d}_k\in \widebar{\F_q}$ for $0\leq k \leq \mu-m-1$, $\hat{d}_{\mu-m}=1$ and $\hat{d}_l=0$ for $l>\mu-m$.
	These polynomials form a $(\mu-m)$-dimensional affine $\widebar{\F_q}$-space and as we have seen, any such $d(x)$ uniquely defines a $c(x)$ for which \eqref{eq} holds.

	Now consider the case when $f$ and $g$ may have multiple roots.
	Let $f=x^{q^{k_1}}\circ \tilde f$ and $g=x^{q^{l_1}}\circ \tilde g$ where 
	$\tilde f$ and $\tilde g$ have only simple roots. W.l.o.g. assume $l_1 \leq k_1$. We want to find the dimension of the solutions of
	\[d \circ x^{q^{k_1}}\circ \tilde f = c \circ x^{q^{l_1}}\circ \tilde g,\]
	under the given assumptions on the degrees and leading coefficients of $c$ and $d$. Clearly, the multiplicities of the roots of the left hand side and the right hand side have to coincide and hence $c=c' \circ x^{q^{k_1-l_1}}$. 
	Let $\tilde d$ and $\tilde c'$ denote the $q$-polynomials whose coefficients are the $q^{-k_1}$-th roots of the coefficients of $d$ and $c$, respectively. 
	Then the solutions of the previous system correspond to the solutions of \[x^{q^{k_1}}\circ \tilde d \circ \tilde f = x^{q^{k_1}}\circ \tilde c' \circ \tilde g\] and hence to those of 
\[\tilde d \circ \tilde f = \tilde c' \circ \tilde g,\]
where the $q$-degree of $\tilde d$ is $(l-l_1)-(m-l_1)$ and the $q$-degree of $\tilde c'$ is $(k-k_1)-(m-l_1)$. 
The roots of the $q$-polynomials $\tilde f$ and $\tilde g$ are simple, thus we can apply the first part of this proof for these polynomials. 
The leading coefficients of $\tilde d$ and $\tilde c'$ are $b_l^{q^{k-m-k_1}}$ and $a_k^{q^{l-m-k_1}}$, respectively; the leading coefficients of $\tilde f$ and $\tilde g$ are $a_k^{q^{-k_1}}$ and $b_l^{q^{-l_1}}$, respectively. Since $b_l^{q^{k-m-k_1}}=b_l^{q^{-l_1}\deg \tilde c'}$ and 
$a_k^{q^{l-m-k_1}}=a_k^{q^{-k_1}\deg \tilde d}$, the conditions on the leading coefficients also hold. 
Note that the $q$-degree of $\gcd (\tilde f, \tilde g)$ is $\mu-l_1$. 
Then the  dimension of the solutions of this system is $(\mu-l_1)-(m-l_1)=\mu-m$.

\end{proof}

\begin{proposition}
	\label{prop}
	Suppose $m \leq \mu$. Then the nullity of the matrix $R_{m,q}(f,g)$ is $\mu-m$.
\end{proposition}
\begin{proof}
Let $f,g,c,d$ be defined as before, then
\[d\circ f-c\circ g=\sum_{i=0}^{l-m} d_i \sum_{j=0}^{k}a_j^{q^i} x^{q^{j+i}}-
\sum_{i=0}^{k-m} c_i \sum_{j=0}^l b_j^{q^i} x^{q^{j+i}}=\]
\[\sum_{i=0}^{k+l-m}\left(\sum_{j=0}^{i}d_{i-j}a_j^{q^{i-j}}-c_{i-j}b_j^{q^{i-j}}\right)x^{q^i}.\]

The $q$-degree of $r:= \gcd(f,g)$ is
$\mu \geq m$ and $r \mid d \circ f - c \circ g$, thus $d$ and $c$ form a solution to
$d \circ f - c \circ g=0$ if
and only if the $q$-degree of $d \circ f - c \circ g$ is less than $m$. In another words, we only have to concentrate on the coefficients of terms with $q$-degree 
$i\in \{m,m+1,\ldots, k+l-m\}$ in $d \circ f - c \circ g$. 

Note that the coefficient of $q^{k+l-m}$ is $d_{l-m}a_k^{q^{l-m}}-c_{k-m}b_l^{q^{k-m}}$ (coefficients out of range are considered to be $0$), which is $0$ because of our assumptions on $c$ and $d$. 
Now let 
\[{\bf v}=(d_{l-m-1},d_{l-m-2},\ldots,d_0,-c_{k-m-1},-c_{k-m-2},\ldots,-c_0)\]
and 
\[{\bf b}=(b_l^{q^{k-m}}a_{k-1}^{q^{l-m}}-a_k^{q^{l-m}}b_{l-1}^{q^{k-m}},\ldots,b_l^{q^{k-m}}a_{2m-l}^{q^{l-m}}-a_k^{q^{l-m}}b_{2m-k}^{q^{k-m}}).\]
We claim that
\begin{equation}
\label{final}
{\bf v}R_{m,q}(f,g)=-{\bf b}
\end{equation}
holds if and only if 
\begin{equation}
\label{st}
\sum_{j=0}^{i}d_{i-j}a_j^{q^{i-j}}-c_{i-j}b_j^{q^{i-j}}=0
\end{equation}
for all $m \leq i \leq k+l-m-1$. To see this we show that the $(k+l-2m-t)$-th coordinates in the vectors at the left and right hand side of \eqref{final} coincide if and only if \eqref{st} holds with $i=m+t$. 
Indeed, in
\begin{equation}
\label{st2}
\sum_{j=0}^{m+t}d_{m+t-j}a_j^{q^{m+t-j}}-c_{m+t-j}b_j^{q^{m+t-j}}
\end{equation}
$d_{m+t-j}\neq 0$ only if $j\in\{m+t,m+t-1,\ldots,2m+t-l\}$ and 
$c_{m+t-j}\neq 0$ only if $j \in  \{m+t,m+t-1,\ldots,2m+t-k\}$. Thus, after changing indices in the summation, \eqref{st2} equals 
\begin{equation}
\label{st3}
\sum_{j=0}^{l-m}d_{l-m-j}a_{2m+t-l+j}^{q^{l-m-j}}-
\sum_{j=0}^{k-m}c_{k-m-j}b_{2m+t-k+j}^{q^{k-m-j}}.
\end{equation}
Since $d_{l-m}=b_l^{q^{k-m}}$ and $c_{k-m}=a_k^{q^{l-m}}$, the $(k+l-2m-t)$-th coordinates on the left and right hand side of \eqref{final} coincide if and only if 
\[\sum_{j=0}^{l-m-1}d_{l-m-1-j}a_{2m-l+1+t+j}^{q^{l-m-1-j}}
-\sum_{s=0}^{k-m-1}c_{k-m-1-s}b_{2m-k+1+t+j}^{q^{k-m-1-s}}=\]
\[d_{l-m}a_{2m-l+t}^{q^{l-m}}-c_{k-m}b_{2m-k+t}^{q^{k-m}},\]
and this happens if and only if \eqref{st3} equals zero.

Thus the dimension of the kernel of the $\widebar{\F_q}$-linear transformation of $\widebar{\F_q}^{k+l-2m}$ defined by ${\bf x}\mapsto{{\bf x}R_{m,q}(f,g)}$ is the same as the dimension of the set of solutions of \eqref{eq} and this finishes the proof. 
\end{proof}

\begin{corollary}
\label{thm}
Let $f$ be a $q$-polynomial over $\F_{q^n}$ and put $g(x)=x^{q^n}-x$. 
Then $\dim_q (\ker f)=\mu$ if and only if
\begin{equation}
|R_{0,q}(f,g)|=|R_{1,q}(f,g)|=\ldots=|R_{\mu-1,q}(f,g)|=0
\end{equation}
and $|R_{\mu,q}(f,g)|\neq 0$.
\end{corollary}

As an illustration, the $(n+k)\times (n+k)$ matrix $R_{0,q}(f,g)$ in the particular case when $g(x)=x^{q^n}-x$ and $f(x)=\sum_{i=0}^k a_i x^{q^i}$ has the following form:

\[\begin{pmatrix}
a_k^{q^{n-1}} & a_{k-1}^{q^{n-1}} & \ldots  & a_0^{q^{n-1}} & 0 & \ldots &0 & 0 & \ldots & 0\\
0 & a_k^{q^{n-2}} & \ldots & a_1^{q^{n-2}} & a_0^{q^{n-2}} & \ldots & 0 & 0& \ldots & 0 \\
\vdots & \vdots & \vdots & \vdots & \vdots & \vdots & \vdots & \vdots & \vdots &\vdots\\
0 & 0 &  \ldots &  0 & 0 & \ldots & a_k & a_{k-1} &\ldots & a_0\\
1 & 0 & \ldots  & 0 & 0 & \ldots & 0 & -1 &\ldots &0\\
0 & 1 & \ldots  & 0 & 0 & \ldots & 0 & 0& \ldots & 0 \\
\vdots & \vdots & \vdots & \vdots & \vdots & \vdots & \vdots & \vdots & \vdots & \vdots \\
0 & 0 & \ldots &  0 & 0 & \ldots & 0 & 0 & \ldots &-1 \\
\end{pmatrix}.\]
The matrix $R_{m,q}(f,g)$ can be obtained from $R_{0,q}(f,g)$ by removing its first and last $m$ columns and its first $m$ rows together with the $(n+1)$-th, $(n+2)$-th, \ldots, $(n+m)$-th rows.

Let  $\tilde{f}(x)=\sum_{i=0}^{k-1} a_i x^{q^i}$ and $g(x)=x^{q^n}-x$. If we substitute $a_k=0$ in $R_{m,q}(f,g)$, then its determinant equals either $|R_{m,q}(\tilde{f},g)|$ or $-|R_{m,q}(\tilde{f},g)|$. This argument can be iterated and hence one can use Corollary  \ref{thm} even if the $q$-degree of $f$ is not known, by considering the $(2n-1-2m)\times (2n-1-2m)$ $m$-th scalar $q$-subresultants of $\sum_{i=0}^{n-1} a_i x^{q^i}$ and $g(x)$. 

\section{A connection with Dickson matrices}
\label{Dickson}

In this section we prove Theorem \ref{conn} but before that we need some preparation.

\begin{result}[Schur's determinant identity, \cite{Schur}]
	\label{Schur}
	Consider the square matrix 
	\[M:=
	\begin{pmatrix}
	X & Y \\
	Z & W \\
	\end{pmatrix},
	\]
	where $W$ is also square and invertible. Then $\det(M)=\det(W)\det(X-YW^{-1}Z)$.
\end{result}

\begin{corollary}
\label{cor}
	Consider the square matrices
	\[M:=
	\begin{pmatrix}
	A & B & C \\
	I_l & O & -I_l \\
	\end{pmatrix},
	\]
	\[N:=
	\begin{pmatrix}
	B & A+C \\
	\end{pmatrix},
	\]
	where $A$ and $C$ are $k\times l$ matrices, $B$ is $k \times (k-l)$, $I_l$ denotes the $l\times l$ identity matrix and $O$ is the $l\times (k-l)$ zero matrix. 
	Then $\det(M)=(-1)^{l(k-l+1)}\det(N)$.
\end{corollary}
\begin{proof}
	Result \ref{Schur} with $X= \begin{pmatrix}
	A & B \\
	\end{pmatrix}$, $Y=C$, $Z=\begin{pmatrix}
	I_l & O \\
	\end{pmatrix}$ and $W=-I_l$ gives
	\[\det(M)=\det(-I_l)\det\left(\begin{pmatrix}
	A & B \\
	\end{pmatrix}+C\begin{pmatrix}
	I_l & O \\
	\end{pmatrix}\right)=(-1)^l\det\begin{pmatrix}
	A+C & B \\
	\end{pmatrix}.\]
	The result follows since $N$ can be obtained from $\begin{pmatrix}
	A+C & B \\
	\end{pmatrix}$ by $l(k-l)$ column changes.
\end{proof}

Let us introduce the abbreviation 
\[R_{m}(f):=R_{m,q}(f,g),\] 
where $g(x)=x^{q^n}-x$ and $f(x)=\sum_{i=0}^{n-1}a_ix^{q^i}$ for some $a_i\in \F_{q^n}$. 

\begin{lemma}
\label{lem}
$|D_m(f)|=|R_{m}(f)|$.
\end{lemma}
\begin{proof}
Note that $D_{n-1}(f)=R_{n-1}(f)=(a_{n-1})$, so we may assume $m<n-1$. 
Let $T_k$ denote the $k\times k$ anti-diagonal matrix whose non-zero entries equal to one and let $I_k$ denote the $k\times k$ identity matrix. By $O$ we will always denote a zero matrix whose dimension will be clear from the context. We distinguish two cases. 

If $m \geq (n-1)/2$, then $2n-1-2m \leq n$ and hence 
$R_m(f)$ has the form: 
\[\begin{pmatrix} 	A & B \\ 
I_{n-1-m} & O \\ \end{pmatrix},\]
where $B=T_{n-m}D_m(f) T_{n-m}$.
We have 
\[\left|\begin{pmatrix} 	A & B \\ 
I_{n-1-m} & O \\ \end{pmatrix}\right|=(-1)^{(n-m-1)(n-m)}\left|\begin{pmatrix} 	B & A \\ 
O & I_{n-1-m} \\ \end{pmatrix}\right|,\]
 and hence by Result \ref{Schur}
\[|R_m(f)|=|B|=|D_m(f)|.\]

If $m < (n-1)/2$, then first consider the last $m$ rows of $R_m(f)$: for $k\in \{0,1,\ldots,m-1\}$ the $(2n-2m-1-k)$-th row of $R_m(f)$ contains only one non-zero entry, namely, a $1$ at position $n-1-m-k$. Then it is easy to see by row expansion applied to the last $m$ rows that: 
\[(-1)^{(n-1)m}|R_m(f)|=\left| 	\begin{pmatrix}
A & B & C \\
I_{n-2m-1} & O & -I_{n-2m-1} \\ \end{pmatrix} \right|,\]
where $A$ and $C$ are $(n-m)\times (n-2m-1)$ matrices and
\[\begin{pmatrix}
	B & A+C \\
\end{pmatrix} =T_{n-m}D_m(f)T_{n-m}.\]
According to Corollary \ref{cor}, 
\[(-1)^{(n-1)m}|R_m(f)|=(-1)^{(n-2m-1)(m+2)}|T_{n-m}D_m(f)T_{n-m}|,\]
which proves the assertion.

\end{proof}

Lemma \ref{lem} immediately yields Theorem \ref{conn}.

\medskip

For some $s$ with $\gcd(s,n)=1$ put $\sigma:=q^s$.
The set of $\sigma$-polynomials over $\F_{q^n}$ is isomorphic to the skew-polynomial ring 
$\F_{q^n}[t,\sigma]$ where $t\alpha=\alpha^{\sigma}t$ for all $\alpha\in \F_{q^n}$. 
Analogies for some of the results of Section \ref{qsubres} should hold in these non-commutative polynomial rings as well. Next we show a generalization of Theorem \ref{conn} for $\sigma$-polynomials.

Consider the $\sigma$-polynomial $f(x):=\sum_{i=0}^{n-1} a_i x^{\sigma^i} \in \F_{q^n}[x]$, which is also a $q$-polynomial. As before, by $\ker f$ we will denote $\gcd(f(x),x^{q^n}-x)$ and similarly to $D(f)$ we define
\[D_{\sigma}(f):=
\begin{pmatrix}
a_0 & a_1 & \ldots & a_{n-1} \\
a_{n-1}^\sigma & a_0^\sigma & \ldots & a_{n-2}^\sigma \\
\vdots & \vdots & \vdots & \vdots \\
a_1^{\sigma^{n-1}} & a_2^{\sigma^{n-1}} & \ldots & a_0^{\sigma^{n-1}}
\end{pmatrix}.\]
We will denote by $D_{m,\sigma}(f)$ the $(n-m)\times (n-m)$  matrix obtained from $D_{\sigma}(f)$ after removing its first $m$ columns and last $m$ rows.
Because of the applications it might be useful to have conditions on other minors of $D_{\sigma}(f)$. In the next corollary we show some results also in this direction.

\begin{corollary}
\label{fin}
If $f(x)=\sum_{i=0}^{n-1} a_i x^{\sigma^i} \in \F_{q^n}[x]$ with $\gcd(s,n)=1$, then 
$\dim_q (\ker f)=\mu$ if and only if
\begin{equation}
|D_{0,\sigma}(f)|=|D_{1,\sigma}(f)|=\ldots=|D_{\mu-1,\sigma}(f)|=0
\end{equation}
and $|D_{\mu,\sigma}(f)|\neq 0$. 

Index the rows and columns of $D_{\sigma}(f)$ from $0$ to $n-1$. For $0\leq m \leq \dim_q (\ker f)$ if $J,K \subseteq \{0,1,\ldots, n-1\}$ are two sets of $m$ consecutive integers modulo $n$ then let $M_{J,K}(f)$ denote the $(n-m)\times (n-m)$ matrix obtained from $D_{\sigma}(f)$ after removing its rows and columns with indices in $J$ and $K$, respectively. 
Then
\[|M_{J,K}(f)|=0 \Leftrightarrow |D_{m,\sigma}(f)|=0.\]
\end{corollary}
\begin{proof}
Consider $f$ as a $q$-polynomial with $\dim_q(\ker f)=\mu$. 
This happens if and only if $D(f)$ has rank $\mu$.
Recall that rows and columns of $D(f)$ are indexed from $0$ to $n-1$ and let $P$ denote the permutation matrix for which the $i$-th row of $PA$ is the $si$-th row of $A$ (considered modulo $n$).
Then $PAP^{-1}=D_{\sigma}(f)$ and hence the rank of $D_{\sigma}(f)$ is the same as the rank of $D(f)$  (cf. also \cite[Remark 2.3]{Puncturing}).
Note that $D_{\sigma}(f)$ is the Dickson matrix of a $\sigma$-polynomial considered as an $\F_{\sigma}$-linear transformation of $\F_{\sigma^n}$ with kernel a $\mu$-dimensional $\F_{\sigma}$-subspace of $\F_{\sigma^n}$. By Theorem \ref{conn} this happens if and only if the conditions on $|D_{m,\sigma}(f)|$ holds for $0\leq m \leq \mu$.

For the second part take $0\leq m \leq \dim_q (\ker f)$.
Note that for any $\sigma$-polynomial $g(x)=\sum_{i=0}^{n-1}b_ix^{\sigma^i}\in \F_{q^n}[x]$ and for any non-negative integer $t$ the rank of $g(x)$ is the same as 
\begin{enumerate}
\item the rank of $g(x)^{\sigma^{t}}$ considered modulo $x^{q^n}-x$,
\item the rank of $\hat{g}(x):=\sum_{i=0}^{n-1}b_{n-i}^{\sigma^i}x^{\sigma^i}$ (since $D_{\sigma}(g)^T=D_{\sigma}(\hat{g})$, where by $^T$ we denote matrix transposition).
\end{enumerate}

Suppose $J=\{j,j+1,\ldots,j+m-1\}$ and $K=\{k,k+1,\ldots,k+m-1\}$ considered modulo $n$. 
Then $f_1(x):=f(x)^{\sigma^{n-k-m}}$ modulo $x^{q^n}-x$ has the same rank as $f(x)$ and 
$|M_{J,K'}(f_1)|=|M_{J,K}(f)|^{\sigma^{n-k-m}}$ where $K'=\{n-m,m+1,\ldots,n-1\}$.
Then $\hat{f_1}(x)$ has the same rank as $f_1(x)$ and 
$|M_{K',J}(\hat{f_1})|=|M_{J,K'}(f_1)|$. Finally, 
$f_2(x):=\hat{f_1}(x)^{\sigma^{n-j}}$ modulo $x^{q^n}-x$ has the same rank as $\hat{f_1}(x)$ and 
$|M_{K',J'}(f_2)|=|M_{K',J}(\hat{f_1})|^{\sigma^{n-j}}$ where $J'=\{0,1,\ldots,m-1\}$.
By definition $M_{K',J'}(f_2)=D_{m,\sigma}(f_2)$, and hence
\[|D_{m,\sigma}(f_2)|=0 \Leftrightarrow |M_{K',J}(\hat{f_1})|=0 \Leftrightarrow |M_{J,K'}(f_1)|=0 \Leftrightarrow |M_{J,K}(f)|=0.\]
Recall $0\leq m \leq \dim_q (\ker f)$. Since $f_2$ and $f$ has the same rank, it follows from the first part of the assertion that $|D_{m,\sigma}(f_2)|=0 \Leftrightarrow |D_{m,\sigma}(f)|=0$ and this finishes the proof.
\end{proof}

\subsection{Applications}

A $q$-polynomial $f(x)\in \F_{q^n}[x]$ is called \emph{scattered} if $\{f(x)/x \colon x\in \F_{q^n}\setminus \{0\}\}$ (the \emph{set of directions determined by the graph of $f$}) has maximum size, that is $(q^n-1)/(q-1)$.
Put $U_f=\{(x,f(x)) \colon x\in \F_{q^n}\}$, which is an $n$-dimensional $\F_q$-subspace of $\F_{q^n}^2$.  
The \emph{linear set} of $\PG(1,q^n)$ defined by $f$ is the set of projective points 
$L_f:=\{\la (x,f(x)) \ra_{q^n} \colon x\in \F_{q^n}\setminus \{0\}\}$. The \emph{weight} of a point $\la (a,b)\ra_{\F_{q^n}}\in \PG(1,q^n)$ w.r.t. the $\F_q$-subspace $U_f$ is 
$\dim_q \la (a,b)\ra_{\F_{q^n}} \cap U_f$.
The polynomial $f$ is scattered if and only if the points of $L_f$ have weight $1$.
In this case $L_f$ and $U_f$ are called \emph{maximum scattered}. 
This happens if and only if the $\F_q$-linear transformations of $\F_{q^n}$ in the $\F_{q^n}$-subspace $M:=\la x, f(x) \ra_{\F_{q^n}}$ have rank at least $n-1$. Equivalently, $M$ is equivalent to an $\F_{q^n}$-linear maximum rank distance (MRD for short) code of $\F_q^{n\times n}$ with minimum distance $n-1$. 
For more details about these objects and the relations among them we refer to \cite[Section 13.3.6]{John} and the references therein.

\begin{corollary}
	\label{key}
	Consider the $q$-polynomial $f(x)=\sum_{i=0}^{n-1}a_i x^{q^i}\in \F_{q^n}[x]$ and
	with $y$ as a variable consider the matrix
	\[H:=\begin{pmatrix}
	y & a_1 & \ldots & a_{n-1} \\
	a_{n-1}^q & y^q & \ldots & a_{n-2}^q \\
	\vdots & \vdots & \vdots & \vdots \\
	a_1^{q^{n-1}} & a_2^{q^{n-1}} & \ldots & y^{q^{n-1}}
	\end{pmatrix}.\]
The determinant of the $(n-m)\times (n-m)$  matrix obtained from $H$ after removing its first $m$ columns and last $m$ rows is a polynomial $H_m(y)\in \F_{q^n}[y]$. Then the following holds:
\begin{enumerate}
	\item The roots of $H_0(y)$ are in $\F_{q^n}$,
	\item the number of points of weight $\mu$ of $L_f$ w.r.t. $U_f$ is the same as the number of common roots of $H_0(y),H_1(y),\ldots, H_{\mu-1}(y)$ which are not roots of $H_{\mu}(y)$,
	\item in particular $f(x)$ is scattered if and only if $H_0(y)$ and $H_1(y)$ have no common roots.
\end{enumerate}
\end{corollary}
\begin{proof}
	Let $y_0$ be a root of $H_0(y)$.
	Note that Lemma \ref{lem} does not require the coefficients of $f$ to be in $\F_{q^n}$, thus also for $y_0\in \overline{\F}_q$ we have  $0=H_0(y_0)=|R_{0,q}(y_0x+\sum_{i=1}^{n-1}a_ix^{q^i}, x^{q^n}-x)|$ and hence by 
	Theorem \ref{resoriginal2} there exists $x_0\in \F_{q^n}\setminus \{0\}$ such that
	$y_0=-\sum_{i=1}^{n-1}a_ix_0^{q^i-1}$. Here the right-hand side is in $\F_{q^n}$ and hence $y_0\in \F_{q^n}$.
	
	By Theorem \ref{conn} $H_0(y_0)=H_1(y_0)=\ldots=H_{\mu-1}(y_0)=0$ and $H_{\mu}(y_0)\neq 0$ hold if and only if the $q$-polynomial $(y_0-a_0)x+f(x)\in \F_{q^n}[x]$ has nullity $\mu$, equivalently, the point $\la (1,a_0-y_0) \ra_{q^n}$ has weight $\mu$.
	
	The last part follows from the fact that $f$ is scattered if and only if $L_f$ does not have points of weight larger than $1$.
\end{proof}

	In \cite{DanieleMaria} Part 3. of Corollary \ref{key} is used to derive sufficient and necessary conditions for $f(x)=bx^q+x^{q^4}\in \F_{q^6}[x]$ to be a scattered polynomial and to prove \cite[Conjecture 7.5]{withZanella} regarding the number of scattered polynomials of this form.
	
	\medskip

	In \cite{withZhou} the authors study MRD-codes with maximum idealisers, or equivalently, the problem of finding sets of distinct integers $\{t_0,t_1,\ldots, t_k\}$ such that every $\F_q$-linear transformation of $\F_{q^n}$ in the $\F_{q^n}$-subspace $\la x^{q^{t_0}}, x^{q^{t_1}},\ldots, x^{q^{t_k}} \ra_{\F_{q^n}}$ has rank at least $n-k$.
	In \cite[Corollary 3.6]{withZhou} it is stated that in $M:=\la x,x^{q},x^{q^{2}},x^{q^{4}}\ra_{\F_{q^9}}$ one can find an $\F_q$-linear transformation of $\F_{q^9}$ with rank at most $5$ and hence the set of integers $\{0,1,2,4\}$ does not satisfy the above mentioned condition. 
	In \cite{withZhou} this was proved by calculating sixteen $6\times 6$ submatrices of 
	$D(f)$, where $f(x)=-x+(1+c^{-q})x^q+cx^{q^2}-x^{q^4}$ and $c\in \F_{q^9}$ satisfies certain conditions, and by proving that each of them has zero determinant. 
	According to Theorem \ref{conn} the same result follows also by calculating only 
	$|D_0(f)|$, $|D_1(f)|$, $|D_2(f)|$, $|D_3(f)|$ and by proving that all of them are zero.

\subsection*{Acknowledgement}

The author is thankful to 
Tam\'as H\'eger from whom he learned Result \ref{resoriginal} and its proof,  which was adapted to prove Theorem \ref{resoriginal2}.

\bigskip

\noindent Bence Csajb\'ok\\
MTA--ELTE Geometric and Algebraic Combinatorics Research Group\\
ELTE E\"otv\"os Lor\'and University, Budapest, Hungary\\
Department of Geometry\\
1117 Budapest, P\'azm\'any P.\ stny.\ 1/C, Hungary\\
{{\em csajbokb@cs.elte.hu}}


\begin{thebibliography}{pippo}

\bibitem{multivar}
{\sc C.~D'Andrea, T.~Krick, A.~Szanto:} Multivariate subresultants in roots, J. Algebra 302 (2006), 16--36.


\bibitem{DanieleMaria}
{\sc D. Bartoli, B.~Csajb\'ok, M. Montanucci:} On a conjecture about maximum scattered subspaces of $\F_{q^6}\times \F_{q^6}$, manuscript.

\bibitem{Schur}
{\sc R.A.~Brualdi,  H.~Schneider:} Determinantal identities: Gauss, Schur, Cauchy, 
Sylvester, Kronecker, Jacobi, Binet, Laplace, Muir, and Cayley. Linear Algebra
Appl.\ 52/53 (1983), 769--791. 

\bibitem{withZhou}
{\sc B.~Csajb\'ok, G.~Marino, O.~Polverino, Y.~Zhou:}. Maximum rank-distance codes with maximum left and right idealisers. Submitted manuscipt. 
\href{https://arxiv.org/abs/1807.08774}{https://arxiv.org/abs/1807.08774} 


\bibitem{teoremone}
{\sc B.~Csajb\'ok, G.~Marino, O.~Polverino, F.~Zullo:} A characterization of linearized polynomials with maximum kernel. Finite Fields Appl.\ 56 (2019), 109--130.

\bibitem{withZanella}
{\sc B.~Csajb\'ok, G.~Marino, O.~Polverino, C.~Zanella:} A new family of MRD-codes. Linear Algebra Appl. 548 (2018), 203--220. 

\bibitem{withZullo}
{\sc B.~Csajb\'ok, G.~Marino, F.~Zullo:} New maximum scattered linear sets of the projective line, Finite Fields Appl. 54 (2018), 133--150. 

\bibitem{Puncturing}
{\sc B.~Csajb\'ok, A. Siciliano:} Puncturing maximum rank distance codes, 
J. Algebraic Combin. 49 (2019), 507-–534. 

\bibitem{sum}
{\sc J. von zur Gathen, T. Lucking:} Subresultants revisited, Theoretical Computer Science 297 (2003), 199--239. 


\bibitem{Hetamas}
{\sc T.~H\'eger:} Some graph theoretic aspects of finite geometries, PhD Thesis, E\"otv\"os Lor\'and University (2013)
Available online at \url{http://web.cs.elte.hu/~hetamas/publ/HTdiss-e.pdf}

\bibitem{Li}
{\sc Z.~Li:} A Subresultant Theory for Ore Polynomials with Applications, in: Proceedings of the 1998 International Symposium on Symbolic and Algebraic Computation, pages 132--139,  ACM Press, 1998.

\bibitem{John} {\sc G.~McGuire, J.~Sheekey:} A Characterization of the number of roots of linearized and projective polynomials in the field of coefficients. 
Finite Fields Appl.\ 57 (2019), 68--91.

\bibitem{GM}
{\sc G. Menichetti:} Roots of affine polynomials, in: Combinatorics '84, Ann.\ Discrete Math.\ 30 (1986), 303--310.

\bibitem{Ore}
{\sc O. Ore:} On a special class of polynomials, Trans. Amer. Math. Soc. 35(3) (1933), 559--584.

\bibitem{Ore2}
{\sc O. Ore:} Theory of Non-Commutative Polynomials, Annals of Mathematics, Second Series, Vol 34., No. 3 (Jul. 1933), pp. 480--508.

\bibitem{John}
{\sc J. Sheekey:} MRD Codes: Constructions and Connections, in: Combinatorics and Finite Fields: Difference Sets, Polynomials, Pseudorandomness and Applications
Ed. by Schmidt, Kai-Uwe and Winterhof, Arne, Series: Radon Series on Computational and Applied Mathematics 23, De Gruyter 2019.


\bibitem{fa}
{\sc D.S.~Thakur:} Function field arithmetic, World Scientific Publishing, River Edge, NJ, 2004.

\bibitem{ffa}
{\sc B.~Wu, Z.~Liu:} Linearized polynomials over finite fields revisited, Finite Fields Appl. 22 (2013), 79--100.

\end{thebibliography}
\end{document}